\documentclass[11pt]{amsart}

\usepackage{amsfonts,amssymb,amsmath,mathrsfs,graphicx}
\usepackage{float}
\usepackage{epsfig}
\usepackage{mathrsfs}
\usepackage{amsfonts}
\usepackage{amsmath}
\usepackage{amstext}
\usepackage{stmaryrd}
\usepackage{amssymb}
\usepackage{amsthm}
\usepackage{mathrsfs}
\usepackage{amsfonts}
\usepackage{enumerate}
\usepackage{amscd}
\usepackage{indentfirst}
\usepackage{enumerate}
\usepackage{amsmath,amsfonts,amssymb,amsthm}
\usepackage{amsmath,amssymb,amsthm,amscd}
\usepackage{graphicx,mathrsfs}
\usepackage{appendix}

 \usepackage{CJK}
 \usepackage{amsmath}
 \newcommand\relphantom[1]{\mathrel{\phantom{#1}}}

\numberwithin{equation}{section}

\newtheorem{theorem}{Theorem}[section]

\newtheorem{lemma}[theorem]{Lemma}

\theoremstyle{definition}

\newtheorem{definition}{Definition}[section]

\theoremstyle{remark}
\newtheorem{remark}{\bf{Remark}}

\title[Energy and cross-helicity conservation for 3D MHD equations]{Energy and cross-helicity conservation for the three-dimensional ideal MHD equations in bounded domain}

\author[Wang]{Yi Wang}
\address[Yi Wang]{\newline \newline CEMS, HCMS, NCMIS, Academy of Mathematics and Systems Science, Chinese Academy of Sciences, Beijing 100190, P. R. China
\newline and School of Mathematical Sciences, University of Chinese Academy of Sciences, Beijing 100049, P. R. China}
\email{wangyi@amss.ac.cn}

\author[Zuo]{Bijun Zuo}
\address[Bijun Zuo]{\newline Institute of Applied Mathematics, AMSS, Academia Sinica, Beijing 100190, P. R. China
\newline and School of Mathematical Sciences, University of Chinese Academy of Sciences, Beijing 100049, P. R. China}
\email{bjzuo@amss.ac.cn}

\thanks{\textbf{Acknowledgment.} Y. Wang is supported by NSFC grants No. 11671385 and 11688101 and CAS Interdisciplinary Innovation Team}

\begin{document}
%%%%%%%%%%%%%%%%

\date{\today}

\begin{abstract}
	In this paper, we prove the energy and cross-helicity conservation of weak solutions to the three-dimensional ideal MHD equations in bounded domain under the interior Besov regularity conditions which are exactly same as the three-dimensional periodic domain case in \cite{Caflisch 97}, and the boundedness and the Besov-type continuity for both the velocity and magnetic fields near the boundary, which seem crucial for the bounded domain case due to the boundary effect. Note that the Besov-type continuity condition near the boundary is consistent with the interior Besov regularity, which is a new condition we proposed in the present paper.
\end{abstract}

\maketitle \centerline{\date}

%\addcontentsline{toc}{section}{References}
%\newpage

%\tableofcontents
%\section{Introduction}
%\addcontentsline{toc}{section}{Introduction}

\section{Introduction}
Energy conservation is an important issue for weak solutions to both the Euler and the MHD equations.
For the incompressible Euler equations, there are many literatures on the energy conservation since the celebrated Onsager's conjecture proposed by Onsager
\cite{Onsager 49} in 1949. Motivated by the laws of turbulence, Onsager first conjectured that there exists a threshold regularity of the weak solutions beyond which the energy is conserved. More precisely, he guessed that for any weak solution to the 3D Euler equations belonging to the H\"{o}lder space $C^\alpha$, the energy is conserved provided that exponent $\alpha>\frac{1}{3}$, and for any $\alpha<\frac{1}{3}$, there exists a weak solution which dissipates the energy.
Then Eyink \cite{Eyink 94} proved that the energy conservation holds true for any weak solution in a subset solution space $C^\alpha_\ast\subset C^\alpha$ with $C^\alpha_\ast$ being the space of functions satisfying the Fourier property $\sum_{k\in\mathbb{Z}^d}|k^\alpha||\hat{\mathbf{u}}(k)|<\infty$.
Later on, Constantin, E and Titi \cite{Constantin 94} proved the Onsager's conjecture for the energy conservation under the Besov regularity $B_{3,\infty}^\alpha$ with $C^\alpha\subset B_{3,\infty}^\alpha$ for the weak solutions to 3D Euler equations in periodic domain.
Then Duchon and Robert \cite{Duchon 00} proved that the local energy equality holds in the sense of distributions provided that the velocity satisfies integration condition $\int_{T^d}|\mathbf{u}(x+y,t)-\mathbf{u}(x)|^3 dx\leq C(t)|y|\sigma(|y|)$, $\forall y\in \mathbb{T}^d$, where $C$ is an integrable function on $[0,T]$ and $\sigma(a)$ tends to $0$ with $a$.
In 2008, Cheskidov, Constantin, Friedlander and Shvydkoy \cite{Cheskidov 08} further sharpened the previous results by using the well-known Littlewood-paley decomposition.

Note that all of the above results are concerned with the Euler flow in the whole space or periodic domain without boundary. For the 3D incompressible Euler equations in bounded domain,
the Onsager's conjecture for the energy conservation was proved by Bardos and Titi \cite{Bardos 18} recently under the H\"older continuity conditions up to the boundary. Then Drivas and Nguyen \cite{Drivas 18} proved the energy conservation by requiring not only the interior Besov regularity, but also the boundedness and the wall-continuity near the boundary for weak solution.

For the case of the incompressible and inhomogeneous Euler equations with the transport of the fluid density and the incompressible or compressible Navier-Stokes equations with the viscosity, the energy conservation was proved by Yu \cite{Yu C 16}, \cite{Yu C 17}, \cite{Yu C 18}, Chen and Yu \cite{Robin Ming Chen 17}, and Chen, Liang, Wang and Xu \cite{Robin Ming Chen 18}.

In the present paper, we are concerned with the energy and cross-helicity conservation for the weak solution to the three-dimensional ideal magneto-hydrodynamics (MHD) equations in bounded domain, which reads
\begin{equation}\label{1.1}
\begin{cases}
\partial_t\mathbf{u}+(\mathbf{u}\cdot\nabla)\mathbf{u}-(\mathbf{b}\cdot\nabla)\mathbf{b}+\nabla \pi=0,\\
\partial_t\mathbf{b}+(\mathbf{u}\cdot\nabla)\mathbf{b}-(\mathbf{b}\cdot\nabla)\mathbf{u}=0, \quad\quad\quad \rm{in}\;Q_T:=\Omega\times[0,T),\\
\nabla\cdot\mathbf{u}=0,\quad\quad\nabla\cdot\mathbf{b}=0,
\end{cases}
\end{equation}
with the Dirichlet boundary condition
\begin{equation}\label{1.2}
\mathbf{u}\cdot\mathbf{n}|_{\partial\Omega}=\mathbf{b}\cdot\mathbf{n}|_{\partial\Omega}=0.
\end{equation}
Here the spatial variable $x\in\Omega\subset \mathbb{R}^3$ with $\Omega$ being a bounded domain with $C^2$ boundary $\partial \Omega$ and the normal direction $\mathbf{n}$, the time variable $t\in[0,T)$ with any fixed $T>0$, and $\mathbf{u}=\mathbf{u}(x,t):Q_T\rightarrow \mathbb{R}^3$ represents the fluid velocity, $\mathbf{b}=\mathbf{b}(x,t):Q_T\rightarrow \mathbb{R}^3$ is the magnetic field, and $\pi=p+\frac{1}{2}|\mathbf{b}|^2:Q_T\rightarrow \mathbb{R}$ is the magnetic pressure, with $p=p(x,t)$ being the fluid pressure.

Then the conserved quantities for the total energy $\mathcal{E}(t)$ and the cross-helicity $\mathcal{H}(t)$ considered in the present paper are defined by
\begin{equation}\label{1.4}
\mathcal{E}(t):=\int_\Omega |\mathbf{u}(x,t)|^2+|\mathbf{b}(x,t)|^2dx,
\end{equation}
\begin{equation}\label{1.5}
\mathcal{H}(t):=\int_\Omega\mathbf{u}(x,t)\cdot\mathbf{b}(x,t)dx.
\end{equation}
Note that the total energy $\mathcal{E}(t)$ in \eqref{1.4} include the fluid kinetic energy and the magnetic energy, and the cross-helicity $\mathcal{H}(t)$ in \eqref{1.5} measures the degree of the linkage of the vortex and the magnetic flux tubes within the flow from a geometric point of view (see \cite{Yos}).

For the ideal MHD equations \eqref{1.1}, Caflisch, Klapper and Steele \cite{Caflisch 97} proved the energy conservation in a periodic domain with no boundary effect by extending the results \cite{Constantin 94} to the ideal MHD equations in a straightforward manner. Then Kang and Lee \cite{Kang 07} proved the energy and cross-helicity conservation in the whole space by the Littlewood-paley decomposition as in \cite{Cheskidov 08}.
Later on, Yu \cite{Yu Xinwei 09} improved the previous results by exploring the special structure of the nonlinear terms in the ideal MHD equations.

A natural question then is when the energy and cross-helicity conservation hold true for the 3D ideal MHD equations \eqref{1.1} in bounded domain.
Compared with the incompressible Euler equations, the ideal MHD equations have higher nonlinearity due to the couplings of the fluid velocity and the magnetic field.
Moreover, quite different from the whole space or periodic domain case, another new difficulty here is how to control the velocity $\mathbf{u}$, the magnetic field $\mathbf{b}$, the magnetic pressure $\pi$ and their couplings near the physical boundary.
To overcome these difficulties, besides the interior Besov regularity on $\mathbf{u}$ and $\mathbf{b}$, which is same as the 3D periodic domain case in \cite{Caflisch 97}, the boundedness and the Besov-type continuity near the boundary are also crucial for the energy and cross-helicity conservation of the 3D ideal MHD equations \eqref{1.1} in bounded domain. Note that the boundedness of $(\mathbf{u}, \mathbf{b}, \pi)$ near the boundary is motivated by \cite{Drivas 18} for the energy conservation of the incompressible Euler equations in bounded domain and the Besov-type continuity condition near the boundary here seems consistent with the interior Besov regularity, which is a new condition we proposed in the present paper.

\section{Main result}
We first introduce some notations and terminologies which will be used throughout this paper.
For the bounded domain $\Omega\subset \mathbb{R}^3$, we define the following distance function for $x\in\Omega$
 \[d(x):=\inf_{y\in\partial\Omega}|x-y|,\]
 and the following subdomains
\begin{equation}\label{2.1}
\Omega_\varepsilon:=\{x\in\Omega\mid d(x)<\varepsilon\}, \quad\Omega^\varepsilon:=\Omega\setminus\overline{\Omega}_\varepsilon,
\end{equation}
where $\overline{\Omega}_\varepsilon$ denotes the closure of $\Omega_\varepsilon$ in the usual Euclidean topology. Since $\partial\Omega$ is $C^2$, there exists a constant $h_0>0$ which may depend on $\Omega$ satisfying (cf. \cite{Robert})
\begin{enumerate}[($i$)]
\item  $d\in C^1(\overline{\Omega_{h_0}})$;

\item for any $ x\in{\Omega_{h_0}}$, there exists a unique point $m(x)\in\partial\Omega$ such that
\begin{equation}\label{2.2}
d(x)=|x-m(x)|,\quad\quad\nabla d(x)=-\mathbf{n}(m(x)).
\end{equation}
\end{enumerate}

 Let $l,h$ be two positive parameters with $0<l<h$ (In fact, we choose $l=\frac{h}{16}$ in the following),
 and define $\eta_{h,l}:\mathbb R\rightarrow \mathbb R$ to be a smooth function with compact support satisfying
\begin{equation}
\eta_{h,l}(z)=
\begin{cases}
0,&z\in(-\infty,h-l],\\
1,&z\in[h,+\infty).
\end{cases}
\end{equation}
The cut-off function $\theta_{h,l}: \Omega\rightarrow \mathbb R$ is defined by
\begin{equation}
  \theta_{h,l}(x):=\eta_{h,l}(d(x)).
\end{equation}
It is easy to check that $\|\eta'_{h,l}(x)\|_{L^\infty(R)}\leq \frac{C}{l}$, $supp\,\theta_{h,l}\subset\Omega^{h-l}$ and  $supp\,\nabla\theta_{h,l}\subset\Omega_h\setminus\Omega_{h-l}$.

Let $\phi\in C_c^\infty(\mathbb R^3)$ be a non-negative radial function satisfying $supp\,\phi\subset \{x\in \mathbb R^3\mid|x|\leq 1\}$ and $\int_{\mathbb R^3}\phi(x) dx=1$.
For any $f\in L^1_{loc}(\Omega)$, we define its mollification $f^l(x)$ with respect to $l$ by
\begin{equation}\label{2.3}
f^l(x)=(f\ast\phi^l)(x):=\int_{\Omega}\phi^l(x-y)f(y)dy=\int_{|y|<1}\phi^l(y)f(x-y)dy,\,\,x\in \Omega^l,
\end{equation}
where $\phi^l(x)=\frac{1}{l^3}\phi(\frac{x}{l})$ is the standard scaling function.

We then give the definition of Besov space. For $1\leq p\leq \infty$, $0<\alpha<1$, the Besov space $B^\alpha_{p,\infty}(\Omega)$ is defined by
\begin{equation}\label{1-19}
B^\alpha_{p,\infty}(\Omega):=\{f\in L^p(\Omega): \|f\|_{L^p(\Omega)}+\sup_{|y|>0}\frac{\|f(\cdot)-f(\cdot-y)\|_{L^p(\Omega\cap(\Omega+\{y\})}}{|y|^\alpha}<\infty\},
\end{equation}
where  $\Omega+\{y\}:=\{x\in R^3| x=z+y, z\in\Omega\}$ with any $y\not\equiv0$ means the translation of the domain $\Omega$ according to $y$.

Next, we recall the definition of the weak solution to the ideal MHD equations.
\begin{definition}\label{2.1.}
We call the triple $(\mathbf{u},\mathbf{b},\pi)$ a weak solution to the ideal MHD equations \eqref{1.1} if
\begin{enumerate}[($i$).]

\item $(\mathbf{u},\mathbf{b})\in L^\infty((0,T);L^2(\Omega))\times L^\infty((0,T);L^2(\Omega)),  \,\,\pi\in L^1_{loc}(0,T;L^1_{loc}(\Omega))$;

\item for any $\mathbf{\Phi}\in C_c^\infty(\Omega\times(0,T);\mathbb R^3)$, it holds that
\begin{equation}\label{2.7}
\begin{split}
&\int_0^T\int_{\Omega}\left(\mathbf{u}\cdot\partial_t\mathbf{\Phi}
+\mathbf{u}\otimes\mathbf{u}:\nabla\mathbf{\Phi} -\mathbf{b}\otimes\mathbf{b}:\nabla\mathbf{\Phi}+\pi\nabla\cdot\mathbf{\Phi}\right) dxdt=0,
\end{split}
\end{equation}
and
\begin{equation}\label{2.8}
\int_0^T\int_{\Omega}\left(\mathbf{b}\cdot\partial_t\mathbf{\Phi}+ \mathbf{b}\otimes\mathbf{u}:\nabla\mathbf{\Phi} -\mathbf{u}\otimes\mathbf{b}:\nabla\mathbf{\Phi}\right) dxdt=0;
\end{equation}

\item for any $\varphi\in C_c^\infty(\Omega)$, it holds that
\begin{equation}\label{2.9}
\int_{\Omega}\mathbf{u}\cdot\nabla\varphi dx=\int_{\Omega}\mathbf{b}\cdot\nabla\varphi dx=0,
\end{equation}
a.e. $t\in (0,T)$.

\item the boundary condition \eqref{1.2} holds in some weak sense.

\end{enumerate}
\end{definition}

Now we state our main result.
\begin{theorem}\label{1.1.}
Let $(\mathbf{u},\mathbf{b},\pi)$ be a weak solution to the 3D ideal MHD equations \eqref{1.1} in $Q_T$ in the sense of Definition \ref{2.1.}. Assume that
\begin{enumerate}[(i).]

\item(Interior Besov regularity) for any $\Gamma\subset\subset\Omega$, there exist $\alpha_1,\alpha_2\in(0,1)$ such that
\begin{equation}\label{1-1}
\mathbf{u}\in L^3(0,T;B_{3,\infty}^{\alpha_1}(\Gamma)),\,\,
\mathbf{b}\in L^3(0,T;B_{3,\infty}^{\alpha_2}(\Gamma)),
\end{equation}

\item(Boundedness near the boundary) there exists a $\sigma_0>0$ such that
\begin{equation}\label{1-2}
\mathbf{u},\mathbf{b}\in L^3(0,T;L^\infty(\Omega_{\sigma_0})),\,\,
\pi\in L^{3/2}(0,T;L^\infty(\Omega_{\sigma_0})),
\end{equation}

\item(Besov-type continuity near the boundary) $\mathbf{u}$ and $\mathbf{b}$ satisfy the following Besov-type continuity near the boundary

\begin{equation}\label{1-3}
\sup_{h>0}\frac{\|\mathbf{u}(x)\cdot\mathbf{n}(m(x))\|_{L^3(\Omega_h)}}{h^{\alpha_1}}\in L^3(0,T),
\end{equation}
\begin{equation}\label{1-4}
\sup_{h>0}\frac{\|\mathbf{b}(x)\cdot\mathbf{n}(m(x))\|_{L^3(\Omega_h)}}{h^{\alpha_2}}\in L^3(0,T),
\end{equation}
\end{enumerate}
then both the total enegy $\mathcal{E}(t)$
and the cross-helicity $\mathcal{H}(t)$ are constants a.e. in $(0,T)$ provided that both the exponents $\alpha_1$ and $\alpha_2$ in \eqref{1-1}, \eqref{1-3} and \eqref{1-4} satisfty $$\alpha_1>\frac13,\quad \alpha_2>\frac{1}{3}.$$
\end{theorem}

\begin{remark}\label{1.2.}
The interior Besov regularity condition in $(i)$ is exactly same as the one in \cite{Caflisch 97} for the 3D ideal MHD equations in periodic domain. The boundedness of $(\mathbf{u},\mathbf{b},\pi)$ in $(ii)$ is motivated by \cite{Drivas 18} for the incompressible Euler equations in bounded domain, and the Besov-type continuity in $(iii)$ is a new condition we proposed in the present paper for the energy and cross-helicity conservation for the 3D ideal MHD equations in bounded domain, which is consistent with
the interior Besov regularity condition in $(i)$.
\end{remark}

\begin{remark}\label{1.3.}
Theorem \ref{1.1.} still holds true if we replace the condition $(iii)$ by the following wall-normal continuity condition $(iii)'$ near the boundary motivated by \cite{Drivas 18} for the incompressible Euler equations in bounded domain:
\begin{equation}\label{1-5}
\begin{split}
(iii)'.\,\,&lim_{h\rightarrow 0}\sup_{x\in \Omega_h}|\mathbf{u}(x,t)\cdot\mathbf{n}(m(x))|=0\,\,\text{in } L^3(0,T),\\
&lim_{h\rightarrow 0}\sup_{x\in \Omega_h}|\mathbf{b}(x,t)\cdot\mathbf{n}(m(x))|=0\,\,\text{in } L^3(0,T),
\end{split}
\end{equation}
Note that the two conditions $(iii)$ and $(iii)'$ are independent each other, that is, the condition $(iii)'$ can not be derived from the condition $(iii)$ and vice versa.
\end{remark}

\begin{remark}\label{1.5.}
If we assume that $\mathbf{u}\in L^3((0,T);C^{\alpha_1}(\overline{\Omega})), \mathbf{b}\in L^3((0,T);C^{\alpha_2}(\overline{\Omega}))$ and $\mathbf{u}\cdot\mathbf{n}\big|_{\partial\Omega}=\mathbf{b}\cdot\mathbf{n}\big|_{\partial\Omega}=0,$ as in Bardos and Titi [1] for the incompressible Euler equations, then
the condition $(i)$, $(ii)$, $(iii)$ are satisfied obviously, thus the energy and cross-helicity conservation hold true.
\end{remark}

\begin{remark}\label{1.6.}
The Besov-type continuity condition $(iii)$ can be viewed as $\mathbf{u}$ and $\mathbf{b}$ satisfy the Dirichlet boundary condition in some weak sense. In particular, if $\mathbf{u},\mathbf{b}\in C(\overline{\Omega})$ additionally, then the condition $(iii)$ implies  $\mathbf{u}\cdot\mathbf{n}\big|_{\partial\Omega}=\mathbf{b}\cdot\mathbf{n}\big|_{\partial\Omega}=0$ in \eqref{1.2}.
\end{remark}

\section{Proof of Theorem \ref{1.1.}}
In this section, we give the detailed proof of our main result Theorem \ref{1.1.}. Hereafter, $C$ is a generic positive constants independent of $l$ and $h$. First we need two useful lemmas.

\begin{lemma}\label{2.2.}
Let $\alpha\in(0,1), k\in \mathbb Z^+$ and $f\in{B}^\alpha_{3,\infty}(\Omega)$, then the following estimates hold true
\begin{equation}\label{2.1300}
\sup_{|y|\leq l}\|f(\cdot)-f(\cdot-y)\|_{L^3(\Omega^{h-l})}\leq Cl^{\alpha}\|f\|_{{B}^\alpha_{3,\infty}(\Omega^{\frac{h}{2}})},
\end{equation}
\begin{equation}\label{2.13}
\|f-f^l\|_{L^3(\Omega^{h-l})}\leq Cl^\alpha\|f\|_{{B}^\alpha_{3,\infty}(\Omega^{\frac{h}{2}})},
\end{equation}
\begin{equation}\label{2.14}
\|D^k f^l\|_{L^3(\Omega^{h-l})}\leq Cl^{\alpha-k}\|f\|_{{B}^\alpha_{3,\infty}(\Omega^{\frac{h}{2}})}.
\end{equation}
\end{lemma}

\begin{proof}
 For the whole space case $\mathbb R^3$, the estimates \eqref{2.1300}-\eqref{2.14} can be found in \cite{Constantin 94}. Here the proof for bounded domain case is similar. In fact, by the definition of Besov spaces, we have
\begin{equation}\label{000}
\begin{split}
\|f-f^l\|_{L^3(\Omega^{h-l})}
&=\|\int_{|y|\leq l}\phi^l(y)(f(x)-f(x-y))dy\|_{L^3(\Omega^{h-l})}\\
&\leq\sup_{|y|\leq l}\|f(\cdot)-f(\cdot-y)\|_{L^3(\Omega^{h-l})}
\leq\sup_{|y|\leq l}\|f(\cdot)-f(\cdot-y)\|_{L^3(\Omega^{\frac{h}{2}}\cap(\Omega^{\frac{h}{2}}+\{y\}))}\\
&\leq Cl^{\alpha}\|f\|_{{B}^\alpha_{3,\infty}(\Omega^{\frac{h}{2}})},\\
\end{split}
\end{equation}
\begin{equation}\label{002}
\begin{split}
\|D^k f^l\|_{L^3(\Omega^{h-l})}
&=\|\int_{|y|\leq l}D_y^k\phi^l(y)f(x-y)dy\|_{L^3(\Omega^{h-l})}\\
&=\|\int_{|y|\leq l}l^{-k-3}D_y^k \phi(\frac{y}{l})f(x-y)dy\|_{L^3(\Omega^{h-l})}\\
&=\|\int_{|y|\leq l}l^{-k-3}D_y^k \phi(\frac{y}{l})(f(x-y)-f(x))dy\|_{L^3(\Omega^{h-l})}\\
&\leq\sup_{|y|\leq l}\|f(\cdot-y)-f(\cdot)\|_{L^3(\Omega^{h-l})}\int_{|y|\leq l}l^{-k-3}|D_y^k\phi({\frac{y}{l}})|dy\\
&\leq  Cl^{\alpha-k}\|f\|_{{B}^\alpha_{3,\infty}(\Omega^{\frac{h}{2}})},
\end{split}
\end{equation}
where in \eqref{002} we used Minkowski's inequality and the fact $\int_{|y|\leq l}l^{-k-3}D_y^k\phi({\frac{y}{l}})dy=0$.
\end{proof}

\begin{lemma}\label{2.2.0}
Let $\mathbf{w}:\Omega\times[0,T)\rightarrow\mathbb R^3$ be a vector field satisfying
$\mathbf{w}\in L^3(0,T;B_{3,\infty}^{\alpha}(\Omega'))$ for any $\Omega'\subset\subset\Omega$ and
\begin{equation}\label{2-1}
\sup_{h>0}\frac{\|\mathbf{w}(x,t)\cdot\mathbf{n}(m(x))\|_{L^3(\Omega_h)}}{h^{\alpha}}\in L^3(0,T),
\end{equation}
then
\begin{equation}\label{2-2}
\left\|\mathbf{w}^l(x,t)\cdot\nabla\theta_{h,l}\right\|_{L^3\left(0,T; L^3(\Omega_h\setminus\Omega_{h-l})\right)}
\leq Cl^{-1}\left(l^{\alpha}\|\mathbf{w}\|_{L^3(0,T;B_{3,\infty}^\alpha(\Omega^{\frac{h}{2}}))}+h^\alpha\right).
\end{equation}
\end{lemma}

\begin{proof}
For any $x\in\Omega_h\setminus\Omega_{h-l}$, by the definition of the cut-off function $\theta_{h,l}(x)$, one has
\begin{equation}\label{2-103}
\nabla\theta_{h,l}(x)=-\eta'_{h,l}\left(d(x)\right)\mathbf{n}\left(m(x)\right).
\end{equation}
By using the Minkowski inequality, we obtain
\begin{equation}\label{2-104}
\begin{array}{ll}
\displaystyle \left\|\mathbf{w}^l(x,t)\cdot\nabla\theta_{h,l}\right\|_{L^3(0,T;L^3(\Omega_h\setminus\Omega_{h-l}))}\\
\displaystyle =\left\|\int_{|y|\leq l}\phi^l(y)\mathbf{w}(x-y,t)\cdot\eta'_{h,l}\left(d(x)\right)\mathbf{n}\left(m(x)\right)dy\right\|_{L^3(0,T;L^3(\Omega_h\setminus\Omega_{h-l}))}\\
\displaystyle  \leq \left\|\eta'_{h,l}\left(d(x)\right)\right\|_{L^\infty(\Omega_h\setminus\Omega_{h-l})}
\Bigg(\|\mathbf{w}(x,t)\cdot\mathbf{n}(m(x))\|_{L^3(0,T;L^3(\Omega_h\setminus\Omega_{h-l}))}\\
\displaystyle  \relphantom{=}+\int_{|y|\leq l}\phi^l(y)\|\left(\mathbf{w}(x-y,t)-\mathbf{w}(x,t)\right)\cdot\mathbf{n}\left(m(x)\right)\|_{L^3(0,T;L^3(\Omega_h\setminus\Omega_{h-l}))} dy\Bigg)\\
\displaystyle  \leq Cl^{-1}\Bigg(\left\|\mathbf{w}(x,t)\cdot\mathbf{n}(m(x))\right\|_{L^3(0,T;L^3(\Omega_h\setminus\Omega_{h-l}))}\\
\displaystyle \relphantom{=}+\sup_{|y|\leq l}\left\|\mathbf{w}(x-y,t)-\mathbf{w}(x,t)\right\|_{L^3(0,T;L^3(\Omega^{\frac{h}{2}}\cap(\Omega^{\frac{h}{2}}+\{y\})))}\Bigg)\\
\displaystyle \leq Cl^{-1}\left(l^{\alpha}\|\mathbf{w}\|_{L^3(0,T;B_{3,\infty}^\alpha(\Omega^{\frac{h}{2}}))}+h^\alpha\right).
\end{array}
\end{equation}
Thus Lemma \ref{2.2.0} is proved.
\end{proof}

Now we are ready to prove Theorem \ref{1.1.} by the following five steps.

\underline{\bf{Step 1}}. Given $x\in\Omega^l$, choosing the test functions $\Phi(y,t)=\left(\phi^l(x-y)\chi(t),0,0\right)^{t}$,\\
                                                  $\left(0,\phi^l(x-y)\chi(t),0\right)^{t}$,
                                           and $\left(0,0,\phi^l(x-y)\chi(t)\right)^{t}$ with $\chi(t)\in C_0^\infty\left((0,T)\right)$ in \eqref{2.7} and \eqref{2.8} respectively,
we obtain for $i=1,2,3,$
\begin{equation}\label{3-101}
\begin{split}
&\int_0^t\int_{\Omega}u_i\phi^l(x-y)\chi'(t)dydt+\int_0^t\int_{\Omega}u_iu_j\left(\partial_{y_j}\phi^l(x-y)\right)\chi(t)dydt\\
&-\int_0^t\int_{\Omega}b_ib_j\left(\partial_{y_j}\phi^l(x-y)\right)\chi(t)dydt
+\int_0^t\int_{\Omega}\pi\left(\partial_{y_i}\phi^l(x-y)\right)\chi(t)dydt=0,
\end{split}
\end{equation}
and
\begin{equation}\label{3-1001}
\begin{split}
&\int_0^t\int_{\Omega}b_i\phi^l(x-y)\chi'(t)dydt+\int_0^t\int_{\Omega}b_iu_j\left(\partial_{y_j}\phi^l(x-y)\right)\chi(t)dydt\\
&-\int_0^t\int_{\Omega}u_ib_j\left(\partial_{y_j}\phi^l(x-y)\right)\chi(t)dydt=0.
\end{split}
\end{equation}
By the definition of mollification, it is easy to get
\begin{equation}\label{3-10-1}
\int_0^t\mathbf{u}^l\chi'(t)dt-\int_0^tdiv(\mathbf{u}\otimes\mathbf{u})^l\chi(t)dt+\int_0^tdiv(\mathbf{b}\otimes\mathbf{b})^l \chi(t)dt
-\int_0^t\nabla\pi^l\chi(t)dt=0,
\end{equation}
and
\begin{equation}\label{3-10-2}
\int_0^t\mathbf{b}^l\chi'(t)dt-\int_0^tdiv(\mathbf{b}\otimes\mathbf{u})^l\chi(t)dt+\int_0^tdiv(\mathbf{u}\otimes\mathbf{b})^l \chi(t)dt=0.
\end{equation}
Multiplying the equality \eqref{3-10-1} by $\theta_{h,l}\mathbf{u}^l$ and \eqref{3-10-2} by $\theta_{h,l}\mathbf{b}^l$, and then summing the resulted equalities together and integrating over $\Omega^l$ with respect to $x$, we derive
\begin{equation}\label{3-102}
\int_0^t\left(\int_{\Omega^l}\theta_{h,l}\left|\mathbf{u}^l\right|^2dx\right)\chi'(t)dt
+\int_0^t\left(\int_{\Omega^l}\theta_{h,l}\left|\mathbf{b}^l\right|^2dx\right)\chi'(t)dt
+\int_0^t I(t)\chi(t)dt=0,
\end{equation}
where
\begin{equation}\label{3.10}
\begin{split}
I(t)&=-\int_{\Omega^l}div(\mathbf{u}\otimes\mathbf{u})^l\cdot\left(\theta_{h,l}\mathbf{u}^l\right)dx
+\int_{\Omega^l}div(\mathbf{b}\otimes\mathbf{b})^l\cdot\left(\theta_{h,l}\mathbf{u}^l\right)dx\\
&\relphantom{=}-\int_{\Omega^l}div(\mathbf{b}\otimes\mathbf{u})^l\cdot\left(\theta_{h,l}\mathbf{b}^l\right)dx
+\int_{\Omega^l}div(\mathbf{u}\otimes\mathbf{b})^l\cdot\left(\theta_{h,l}\mathbf{b}^l\right)dx
-\int_{\Omega^l}\nabla\pi^l\cdot\left(\theta_{h,l}\mathbf{u}^l\right) dx\\
&:=I_1(t)+I_2(t)+I_3(t)+I_4(t)+I_5(t).
\end{split}
\end{equation}

First we claim that
\begin{equation}\label{3.11}
\int_{\Omega^l}\theta_{h,l}\left|\mathbf{u}^l\right|^2dx \rightarrow  \int_{\Omega}|\mathbf{u}|^2dx, \quad\quad \int_{\Omega^l}\theta_{h,l}\left|\mathbf{b}^l\right|^2dx \rightarrow  \int_{\Omega}|\mathbf{b}|^2dx,
\end{equation}
in $L^1(0,T)$ as $l\rightarrow 0$. In fact,
\begin{equation}\label{3.12}
\begin{split}
&\int_0^T\left|\int_{\Omega^l}\theta_{h,l}\left|\mathbf{u}^l\right|^2dx-\int_{\Omega}|\mathbf{u}|^2dx\right|dt\\
&\leq\int_0^T\left|\int_{\Omega^l}\theta_{h,l}\left|\mathbf{u}^l\right|^2dx-\int_{\Omega^l}\theta_{h,l}|\mathbf{u}|^2dx\right|dt
+\int_0^T\left|\int_{\Omega^l}\theta_{h,l}|\mathbf{u}|^2dx-\int_{\Omega^l}|\mathbf{u}|^2dx\right|dt\\
&\relphantom{=}+\int_0^T\left|\int_{\Omega^l}|\mathbf{u}|^2dx-\int_{\Omega}|\mathbf{u}|^2dx\right|dt\\
&:=J_1+J_2+J_3.
\end{split}
\end{equation}
For $J_1(t)$, we have
\begin{equation}\label{3-1-1}
\begin{split}
J_1=&\int_0^T\left|\int_{\Omega^l}\theta_{h,l}\left|\mathbf{u}^l\right|^2dx-\int_{\Omega^l}\theta_{h,l}|\mathbf{u}|^2dx\right|dt\\
\leq& \int_0^T\int_{\Omega^{h-l}}\left|\left|\mathbf{u}^l\right|^2-|\mathbf{u}|^2\right|dxdt\\
\leq& \left(\int_0^T\int_{\Omega^{h/2}}\left|\mathbf{u}^l+\mathbf{u}\right|^2dxdt\right)^{1/2} \left(\int_0^T\int_{\Omega^{h/2}}\left|\mathbf{u}^l-\mathbf{u}\right|^2dxdt\right)^{1/2}.
\end{split}
\end{equation}
On the one hand, we have
\begin{equation}\label{3-1-2}
\begin{split}
\left(\int_0^T\int_{\Omega^{h/2}}\left|\mathbf{u}^l+\mathbf{u}\right|^2dxdt\right)^{1/2}
\leq&\left(\int_0^T\int_{\Omega^{h/2}}2\left(\left|\mathbf{u}^l\right|^2+|\mathbf{u}|^2\right)dxdt\right)^{1/2}\\
\leq& \left(\int_0^T\int_{\Omega}4|\mathbf{u}|^2dxdt\right)^{1/2}
<\infty.
\end{split}
\end{equation}
On the other hand, by using the dominant convergence theorem, we have
\begin{equation}\label{3-1-3}
\left(\int_0^T\int_{\Omega^{h/2}}\left|\mathbf{u}^l-\mathbf{u}\right|^2dxdt\right)^{1/2}\rightarrow 0, \,\,{\rm as}\, l\rightarrow 0.
\end{equation}
Substituting \eqref{3-1-2} and \eqref{3-1-3} into \eqref{3-1-1}, one has
\[J_1\rightarrow 0,\,\,{\rm as}\,l\rightarrow 0.\]
For $J_2(t)$, recalling that $l=\frac{h}{16}$, we have
\begin{equation}\label{3-1-4}
\begin{split}
\int_0^T\left|\int_{\Omega^l}\theta_{h,l}|\mathbf{u}|^2dx-\int_{\Omega^l}|\mathbf{u}|^2dx\right|dt
=& \int_0^T\left|\int_{\Omega_h}(\theta_{h,l}-1)|\mathbf{u}|^2dx\right|dt\\
\leq& 2\int_0^T\int_{\Omega_h}|\mathbf{u}|^2dxdt\rightarrow 0, \,\, {\rm as} \,\ l\rightarrow 0.
\end{split}
\end{equation}
For $J_3(t)$, we have
\begin{equation}\label{3-1-5}
\int_0^T\left|\int_{\Omega^l}|\mathbf{u}|^2dx-\int_{\Omega}|\mathbf{u}|^2dx\right|dt
=\int_0^T\int_{\Omega_l}|\mathbf{u}|^2dxdt\rightarrow 0,\,\, {\rm as} \,\ l\rightarrow 0.
\end{equation}

Therefore, to conclude the proof of Theorem \ref{1.1.}, it is sufficient to prove that
\begin{equation}\label{3-1-50}
I(t)\rightarrow 0\,\,{\rm in}\,L^1(0,T), \,\,{\rm as} \,\ l\rightarrow 0.
\end{equation}

\underline{\bf{Step 2}}. In this step, we deal with
$I_1(t)=-\int_{\Omega^l}div(\mathbf{u}\otimes\mathbf{u})^l\cdot\left(\theta_{h,l}\mathbf{u}^l\right)dx$.
Using the divergence-free property of $\mathbf{u}^l$, we can decompose the nonlinear term $I_1(t)$ into two parts
\begin{equation}\label{3-1-6}
\begin{split}
I_1(t)=&-\int_{\Omega^l}div(\mathbf{u}\otimes\mathbf{u})^l\cdot\left(\theta_{h,l}\mathbf{u}^l\right)dx
=\int_{\Omega^l}(\mathbf{u}\otimes\mathbf{u})^l:\nabla\left(\theta_{h,l}\mathbf{u}^l\right)dx\\
=&\int_{\Omega^l}\left(\mathbf{u}^l\otimes\mathbf{u}^l\right):\nabla\left(\theta_{h,l}\mathbf{u}^l\right)dx
+\int_{\Omega^l}\left((\mathbf{u}\otimes\mathbf{u})^l-\left(\mathbf{u}^l\otimes\mathbf{u}^l\right)\right):\nabla\left(\theta_{h,l}\mathbf{u}^l\right)dx\\
:=&I_{11}(t)+I_{12}(t).
\end{split}
\end{equation}

For $I_{11}(t)$, integrating by parts, we have
\begin{equation}\label{3-1-7}
\begin{split}
I_{11}(t)&=\int_{\Omega^l}\left(\mathbf{u}^l\otimes\mathbf{u}^l\right):\nabla\left(\theta_{h,l}\mathbf{u}^l\right)dx
 =\int_{\Omega^l}u_i^lu_j^l\partial_{x_j} \left (\theta_{h,l}u_i^l\right)dx\\
 &=\int_{\Omega^l}\left(u_i^l\right)^2 u_j^l(\partial_{x_j} \theta_{h,l})dx+\int_{\Omega^l}\theta_{h,l}u_j^l\partial_{x_j}\frac{\left(u_i^l\right)^2}{2}dx\\
 &=\int_{\Omega^l}\left(u_i^l\right)^2 u_j^l(\partial_{x_j} \theta_{h,l})dx-\int_{\Omega^l}(\partial_{x_j}\theta_{h,l})u_j^l\frac{\left(u_i^l\right)^2}{2}dx
 -\int_{\Omega^l}\theta_{h,l}\left(\partial_{x_j}u_j^l\right)\frac{\left(u_i^l\right)^2}{2}dx\\
 &=\int_{\Omega^l}\left(\partial_{x_j}\theta_{h,l}\right)u_j^l\frac{\left(u_i^l\right)^2}{2}dx\\
&=\int_{\Omega_h\setminus\Omega_{h-l}}\left(\mathbf{u}^l\cdot\nabla\theta_{h,l}\right)\frac{\left|\mathbf{u}^l\right|^2}{2} dx.
\end{split}
\end{equation}

For $I_{12}(t)$, we have
\begin{equation}\label{3-1-8}
\begin{array}{ll}
\displaystyle I_{12}(t)
\displaystyle=\int_{\Omega_h\setminus\Omega_{h-l}}
\int_{|y|\leq l}\phi^l(y)\big(\mathbf{u}(x-y,t)-\mathbf{u}(x,t)\big)\otimes\left(\mathbf{u}(x-y,t)-\mathbf{u}(x,t)\right)dy:\big(\mathbf{u}^l\otimes\nabla\theta_{h,l}\big) dx\\[5mm]
\displaystyle \qquad\quad +\int_{\Omega^{h-l}}\int_{|y|\leq l}\phi^l(y)\big(\mathbf{u}(x-y,t)-\mathbf{u}(x,t)\big)\otimes\big(\mathbf{u}(x-y,t)-\mathbf{u}(x,t)\big)dy:\big(\theta_{h,l}\nabla\mathbf{u}^l\big)  dx\\[5mm]
\displaystyle\qquad\quad -\int_{\Omega_h\setminus\Omega_{h-l}}\big(\mathbf{u}(x,t)-\mathbf{u}^l(x,t)\big)\otimes\big(\mathbf{u}(x,t)-\mathbf{u}^l(x,t)\big):
\big(\mathbf{u}^l\otimes\nabla\theta_{h,l}\big) dx\\[5mm]
\displaystyle\qquad\quad -\int_{\Omega^{h-l}}\big(\mathbf{u}(x,t)-\mathbf{u}^l(x,t)\big)\otimes\big(\mathbf{u}(x,t)-\mathbf{u}^l(x,t)\big):\big(\theta_{h,l}\nabla\mathbf{u}^l\big) dx\\[4mm]
\displaystyle\qquad\,\,:=R_{11}(t)+R_{12}(t)+R_{13}(t)+R_{14}(t),
\end{array}
\end{equation}
where we have used the following commutator equality
\begin{equation}\label{3-1-9}
\begin{split}
(fg)^l(x,t)=f^l(x,t) g^l(x,t)&+\int_{|y|\leq l}\phi^l(y)\left(f(x-y,t)-f(x,t)\right)\left(g(x-y,t)-g(x,t)\right) dy\\
&-\left(f(x,t)-f^l(x,t)\right)\left(g(x,t)-g^l(x,t)\right).
\end{split}
\end{equation}
Using the fact that $supp\,\theta_{h,l}\subset\Omega^{h-l}$ and $supp\,\nabla\theta_{h,l}\subset\Omega_h\setminus\Omega_{h-l}$, we then divide the terms in the right hand of \eqref{3-1-8} into two parts: the interior terms containing the cut-off function $\theta_{h,l}$ and the boundary terms containing the gradient of the cut-off function $\nabla\theta_{h,l}$. In fact, $R_{12}(t)$ and $R_{14}(t)$ are the interior terms, while $I_{11}(t)$,
$R_{11}(t)$ and $R_{13}(t)$ are the boundary terms. We then estimate them one by one.

For the interior term $R_{12}(t)$, by using \eqref{2.13} and Fubini's theorem, we have
\begin{equation}\label{3-1-10}
\begin{array}{ll}
\displaystyle\quad \int_0^T|R_{12}(t)|dt\\
\displaystyle=\int_0^T\Bigg|\int_{\Omega^{h-l}}\Big(\int_{|y|\leq l}\phi^l(y)(\mathbf{u}(x-y,t)-\mathbf{u}(x,t))\otimes(\mathbf{u}(x-y,t)-\mathbf{u}(x,t))dy\Big)
:\left(\theta_{h,l}\nabla\mathbf{u}^l\right)  dx\Bigg|dt\\[5mm]
\displaystyle\leq C\sup_{|y|\leq l}\|\mathbf{u}(x-y,t)-\mathbf{u}(x,t)\|_{L^3(0,T;L^3(\Omega^{h-l}))}^2
\left\|\nabla\mathbf{u}^l\right\|_{L^3(0,T;L^3(\Omega^{h-l}))}\\
\displaystyle\leq Cl^{3\alpha_1-1}\left\|\mathbf{u}\right\|_{L^3(0,T;B_{3,\infty}^{\alpha_1}(\Omega^{\frac{h}{2}}))}^3.
\end{array}
\end{equation}

Similarly,
\begin{equation}\label{3-1-11}
\begin{split}
\int_0^T|R_{14}(t)|dt&=\int_0^T\left|\int_{\Omega^{h-l}}\left(\mathbf{u}(x,t)-\mathbf{u}^l(x,t)\right)\otimes\left(\mathbf{u}(x,t)-\mathbf{u}^l(x,t)\right)
:\left(\theta_{h,l}\nabla\mathbf{u}^l\right) dx\right|dt\\
&\leq C\left\|\mathbf{u}(x,t)-\mathbf{u}^l(x,t)\right\|_{L^3(0,T;L^3(\Omega^{h-l}))}^2 \left\|\nabla\mathbf{u}^l\right\|_{L^3(0,T;L^3(\Omega^{h-l}))}\\
&\leq Cl^{3\alpha_1-1}\|\mathbf{u}\|_{L^3(0,T;B_{3,\infty}^{\alpha_1}(\Omega^{\frac{h}{2}}))}^3.
\end{split}
\end{equation}

Then we estimate the boundary terms $I_{11}(t)$, $R_{11}(t)$ and $R_{13}(t)$. Using Lemma \ref{2.2.0} and the H\"{o}lder inequality,
\begin{equation}\label{3-1-12}
\begin{split}
&\int_0^T|I_{11}(t)|dt\\
=&\int_0^T\left|\int_{\Omega_h\setminus\Omega_{h-l}}\left(\mathbf{u}^l\cdot\nabla\theta_{h,l}\right)\frac{\left|\mathbf{u}^l\right|^2}{2} dx\right|dt\\
    \leq& C\left\|\mathbf{u}^l(x,t)\cdot\nabla\theta_{h,l}\right\|_{L^3(0,T;L^3(\Omega_h\setminus\Omega_{h-l}))}
    \left\|\mathbf{u}^l\right\|_{L^3(0,T;L^\infty(\Omega_h\setminus\Omega_{h-l}))}^2
      \left|\Omega_h\setminus\Omega_{h-l}\right|^{1-\frac{1}{3}}\\
    \leq& Cl^{-\frac{1}{3}}\left(l^{\alpha_1}\|\mathbf{u}\|_{L^3(0,T;B_{3,\infty}^{\alpha_1}(\Omega^{\frac{h}{2}}))}+h^{\alpha_1}\right)
    \|\mathbf{u}\|_{L^3(0,T;L^\infty(\Omega_{2h}))}^2.
\end{split}
\end{equation}
Similarly,
\begin{equation}\label{3-1-13}
\begin{array}{ll}
\displaystyle\quad \int_0^T|R_{11}(t)|dt\\
\displaystyle=\int_0^T\Bigg|\int_{\Omega_h\setminus\Omega_{h-l}}
\Big(\int_{|y|\leq l}\phi^l(y)(\mathbf{u}(x-y,t)-\mathbf{u}(x,t))\otimes(\mathbf{u}(x-y,t)-\mathbf{u}(x,t))dy\Big)\\
\displaystyle\qquad\quad:\left(\mathbf{u}^l\otimes\nabla\theta_{h,l}\right) dx\Bigg|dt\\[5mm]
\displaystyle\leq Cl^{-1} \sup_{|y|\leq l}\left\|\mathbf{u}(x-y,t)-\mathbf{u}(x,t)\right\|_{L^3(0,T;L^\infty(\Omega_h\setminus\Omega_{h-l}))}
\left|\Omega_h\setminus\Omega_{h-l}\right|^{1-\frac{1}{3}}\\
\displaystyle \qquad\quad\sup_{|y|\leq l}\left\|\mathbf{u}(x-y,t)-\mathbf{u}(x,t)\right\|_{L^3(0,T;L^3(\Omega_h\setminus\Omega_{h-l}))}
\left\|\mathbf{u}^l\right\|_{L^3(0,T;L^\infty(\Omega_h\setminus\Omega_{h-l}))}\\
\displaystyle\leq Cl^{\alpha_1-\frac{1}{s}}\|\mathbf{u}\|_{L^3(0,T;B_{3,\infty}^{\alpha_1}(\Omega^{\frac{h}{2}}))}\|\mathbf{u}\|_{L^3(0,T;L^\infty(\Omega_{2h}))}^2,
\end{array}
\end{equation}
and
\begin{equation}\label{3-1-14}
\begin{split}
&\int_0^T|R_{13}(t)|dt\\
=&\int_0^T\left|\int_{\Omega_h\setminus\Omega_{h-l}}
(\mathbf{u}(x,t)-\mathbf{u}^l(x,t))\otimes(\mathbf{u}(x,t)-\mathbf{u}^l(x,t))
:(\mathbf{u}^l\otimes\nabla\theta_{h,l}) dx\right| dt\\
\leq& Cl^{-1}\left\|\mathbf{u}(x,t)-\mathbf{u}^l(x,t)\right\|_{L^3(0,T;L^\infty(\Omega_h\setminus\Omega_{h-l}))}
\left|\Omega_h\setminus\Omega_{h-l}\right|^{1-\frac{1}{3}}\\
&\left\|\mathbf{u}(x,t)-\mathbf{u}^l(x,t)\right\|_{L^3(0,T;L^3(\Omega_h\setminus\Omega_{h-l}))}
\left\|\mathbf{u}^l\right\|_{L^3(0,T;L^\infty(\Omega_h\setminus\Omega_{h-l}))}\\
\leq& Cl^{\alpha_1-\frac{1}{3}}\|\mathbf{u}\|_{L^3(0,T;B_{3,\infty}^{\alpha_1}(\Omega^{\frac{h}{2}}))}\|\mathbf{u}\|_{L^3(0,T;L^\infty(\Omega_{2h}))}^2.
\end{split}
\end{equation}
Combining the estimates \eqref{3-1-10}-\eqref{3-1-14} together, if $\alpha_1>\frac{1}{3}$, we obtain
\begin{equation}\label{3-1-15}
\int_0^T |I_1(t)| dt\rightarrow 0,\,\,as\,l\rightarrow 0.
\end{equation}

\underline{\bf{Step 3}}. We then deal with the term $I_3(t)$.
\begin{equation}\label{3-1-16}
\begin{split}
I_3(t)=&-\int_{\Omega^l}div(\mathbf{b}\otimes\mathbf{u})^l\cdot\left(\theta_{h,l}\mathbf{b}^l\right)dx
=\int_{\Omega^l}(\mathbf{b}\otimes\mathbf{u})^l:\nabla\left(\theta_{h,l}\mathbf{b}^l\right)dx\\
=&\int_{\Omega^l}\left(\mathbf{b}^l\otimes\mathbf{u}^l\right):\nabla\left(\theta_{h,l}\mathbf{b}^l\right)dx
+\int_{\Omega^l}\left((\mathbf{b}\otimes\mathbf{u})^l-\left(\mathbf{b}^l\otimes\mathbf{u}^l\right)\right):\nabla\left(\theta_{h,l}\mathbf{b}^l\right)dx\\
:=&I_{31}(t)+I_{32}(t).
\end{split}
\end{equation}
Similar as Step 2, we have
\begin{equation}\label{3-1-17}
I_{31}(t)=-\int_{\Omega_h\setminus\Omega_{h-l}}\left(\mathbf{u}^l\cdot\nabla\theta_{h,l}\right)\frac{\left|\mathbf{b}^l\right|^2}{2} dx,
\end{equation}
and
\begin{equation}\label{3-1-18}
\begin{array}{ll}
\displaystyle I_{32}(t)
\displaystyle =\int_{\Omega_h\setminus\Omega_{h-l}}
\int_{|y|\leq l}\phi^l(y)\big(\mathbf{b}(x-y,t)-\mathbf{b}(x,t)\big)\otimes\left(\mathbf{u}(x-y,t)-\mathbf{u}(x,t)\right)dy:\big(\mathbf{b}^l\otimes\nabla\theta_{h,l}\big) dx\\[5mm]
\displaystyle \qquad\quad +\int_{\Omega^{h-l}}\int_{|y|\leq l}\phi^l(y)\big(\mathbf{b}(x-y,t)-\mathbf{b}(x,t)\big)\otimes\big(\mathbf{u}(x-y,t)-\mathbf{u}(x,t)\big)dy:\big(\theta_{h,l}\nabla\mathbf{b}^l\big)  dx\\[5mm]
\displaystyle \qquad\quad -\int_{\Omega_h\setminus\Omega_{h-l}}\big(\mathbf{b}(x,t)-\mathbf{b}^l(x,t)\big)\otimes\big(\mathbf{u}(x,t)-\mathbf{u}^l(x,t)\big):
\big(\mathbf{b}^l\otimes\nabla\theta_{h,l}\big) dx\\[5mm]
\displaystyle \qquad\quad -\int_{\Omega^{h-l}}\big(\mathbf{b}(x,t)-\mathbf{b}^l(x,t)\big)\otimes\big(\mathbf{u}(x,t)-\mathbf{u}^l(x,t)\big):\big(\theta_{h,l}\nabla\mathbf{b}^l\big) dx.\\[4mm]
\end{array}
\end{equation}

Due to the fact that the fluid velocity $\mathbf{u}$ and the magnetic field $\mathbf{b}$ are required to satisfy similar conditions in Theorem \ref{1.1.}, we have the following estimates
\begin{equation}\label{3-1-19}
\int_0^T|I_{31}(t)|dt
    \leq Cl^{-\frac{1}{3}}\left(l^{\alpha_1}\|\mathbf{u}\|_{L^3(0,T;B_{3,\infty}^{\alpha_1}(\Omega^{\frac{h}{2}}))}+h^{\alpha_1}\right)
    \|\mathbf{b}\|_{L^3(0,T;L^\infty(\Omega_{2h}))}^2,
\end{equation}
and
\begin{equation}\label{3-1-20}
\begin{split}
\int_0^T|I_{32}(t)|dt
\leq& C \left[l^{\alpha_1+2\alpha_2-1}\|\mathbf{u}\|_{L^3(0,T;B_{3,\infty}^{\alpha_1}(\Omega^{\frac{h}{2}}))}
\|\mathbf{b}\|_{L^3(0,T;B_{3,\infty}^{\alpha_2}(\Omega^{\frac{h}{2}}))}^2\right.\\
&+\left. l^{-\frac{1}{3}}\left(l^{\alpha_1}\|\mathbf{u}\|_{L^3(0,T;B_{3,\infty}^{\alpha_1}(\Omega^{\frac{h}{2}}))}+h^{\alpha_1}\right)
\|\mathbf{b}\|_{L^3(0,T;L^\infty(\Omega_{2h}))}^2\right].
\end{split}
\end{equation}
Thus, from the estimates \eqref{3-1-19} and \eqref{3-1-20}, if $\alpha_1,\,\alpha_2>\frac{1}{3}$, we have
\begin{equation}\label{3-1-21}
\int_0^T |I_3(t)| dt\rightarrow 0,\,\,as\, l\rightarrow 0.
\end{equation}

\underline{\bf{Step 4}}. For the terms $I_2(t)$ and $I_4(t)$, we can not prove that $I_2(t)$ or $I_4(t)$ vanishes seperately. Fortunately, we are able to show $I_2(t)+I_4(t)\rightarrow0$ in $L^1(0,T)$, which is also crucial in our proof.

 For $I_2(t)$, we have
\begin{equation}\label{3-1-22}
\begin{split}
I_2(t)=&\int_{\Omega^l}div(\mathbf{b}\otimes\mathbf{b})^l\cdot\left(\theta_{h,l}\mathbf{u}^l\right)dx
=-\int_{\Omega^l}(\mathbf{b}\otimes\mathbf{b})^l:\nabla\left(\theta_{h,l}\mathbf{u}^l\right)dx\\
=&-\int_{\Omega^l}\left(\mathbf{b}^l\otimes\mathbf{b}^l\right):\nabla\left(\theta_{h,l}\mathbf{u}^l\right)dx
-\int_{\Omega^l}\left((\mathbf{b}\otimes\mathbf{b})^l-\left(\mathbf{b}^l\otimes\mathbf{b}^l\right)\right):\nabla\left(\theta_{h,l}\mathbf{u}^l\right)dx\\
:=&I_{21}(t)+I_{22}(t).
\end{split}
\end{equation}
After some calculations, we have
\begin{equation}\label{3-1-23}
I_{21}(t)
 =-\int_{\Omega_h\setminus\Omega_{h-l}}\left(\mathbf{b}^l\cdot\mathbf{u}^l\right)\left(\mathbf{b}^l\cdot\nabla\theta_{h,l}\right)dx
 -\int_{\Omega^{h-l}}\left(\mathbf{b}^l\otimes\mathbf{b}^l\right):\left(\theta_{h,l}\nabla\mathbf{u}^l\right)dx,
\end{equation}
and
\begin{equation}\label{3-1-24}
\begin{array}{ll}
\displaystyle I_{22}(t)
\displaystyle =-\int_{\Omega_h\setminus\Omega_{h-l}}
\int_{|y|\leq l}\phi^l(y)\big(\mathbf{b}(x-y,t)-\mathbf{b}(x,t)\big)\otimes\left(\mathbf{b}(x-y,t)-\mathbf{b}(x,t)\right)dy:\big(\mathbf{u}^l\otimes\nabla\theta_{h,l}\big) dx\\[5mm]
\displaystyle \qquad\qquad -\int_{\Omega^{h-l}}\int_{|y|\leq l}\phi^l(y)\big(\mathbf{b}(x-y,t)-\mathbf{b}(x,t)\big)\otimes\big(\mathbf{b}(x-y,t)-\mathbf{b}(x,t)\big)dy:\big(\theta_{h,l}\nabla\mathbf{u}^l\big)  dx\\[5mm]
\displaystyle \qquad\qquad +\int_{\Omega_h\setminus\Omega_{h-l}}\big(\mathbf{b}(x,t)-\mathbf{b}^l(x,t)\big)\otimes\big(\mathbf{b}(x,t)-\mathbf{b}^l(x,t)\big):
\big(\mathbf{u}^l\otimes\nabla\theta_{h,l}\big) dx\\[5mm]
\displaystyle \qquad\qquad +\int_{\Omega^{h-l}}\big(\mathbf{b}(x,t)-\mathbf{b}^l(x,t)\big)\otimes\big(\mathbf{b}(x,t)-\mathbf{b}^l(x,t)\big):\big(\theta_{h,l}\nabla\mathbf{u}^l\big) dx.\\[4mm]
\end{array}
\end{equation}
For $I_4(t)$, we have
\begin{equation}\label{3-1-25}
\begin{split}
I_4(t)=&\int_{\Omega^l}div(\mathbf{u}\otimes\mathbf{b})^l\cdot\left(\theta_{h,l}\mathbf{b}^l\right)dx
=-\int_{\Omega^l}(\mathbf{u}\otimes\mathbf{b})^l:\nabla\left(\theta_{h,l}\mathbf{b}^l\right)dx\\
=&-\int_{\Omega^l}\left(\mathbf{u}^l\otimes\mathbf{b}^l\right):\nabla\left(\theta_{h,l}\mathbf{b}^l\right)dx
-\int_{\Omega^l}\left((\mathbf{u}\otimes\mathbf{b})^l-\left(\mathbf{u}^l\otimes\mathbf{b}^l\right)\right):\nabla\left(\theta_{h,l}\mathbf{b}^l\right)dx\\
:=&I_{41}(t)+I_{42}(t).
\end{split}
\end{equation}

Similarly, we have
\begin{equation}\label{3-1-26}
I_{41}(t)=-\int_{\Omega_h\setminus\Omega_{h-l}}\left(\mathbf{b}^l\cdot\mathbf{u}^l\right)\left(\mathbf{b}^l\cdot\nabla\theta_{h,l}\right)dx
 -\int_{\Omega^{h-l}}\left(\mathbf{u}^l\otimes\mathbf{b}^l\right):\left(\theta_{h,l}\nabla\mathbf{b}^l\right)dx,
\end{equation}
and
\begin{equation}\label{3-1-27}
\begin{array}{ll}
\displaystyle I_{42}(t)
\displaystyle =-\int_{\Omega_h\setminus\Omega_{h-l}}
\int_{|y|\leq l}\phi^l(y)\big(\mathbf{u}(x-y,t)-\mathbf{u}(x,t)\big)\otimes\left(\mathbf{b}(x-y,t)-\mathbf{b}(x,t)\right)dy:\big(\mathbf{b}^l\otimes\nabla\theta_{h,l}\big) dx\\[5mm]
\displaystyle \qquad\qquad-\int_{\Omega^{h-l}}\int_{|y|\leq l}\phi^l(y)\big(\mathbf{u}(x-y,t)-\mathbf{u}(x,t)\big)\otimes\big(\mathbf{b}(x-y,t)-\mathbf{b}(x,t)\big)dy:\big(\theta_{h,l}\nabla\mathbf{b}^l\big)  dx\\[5mm]
\displaystyle \qquad\qquad+\int_{\Omega_h\setminus\Omega_{h-l}}\big(\mathbf{u}(x,t)-\mathbf{u}^l(x,t)\big)\otimes\big(\mathbf{b}(x,t)-\mathbf{b}^l(x,t)\big):
\big(\mathbf{b}^l\otimes\nabla\theta_{h,l}\big) dx\\[5mm]
\displaystyle\qquad\qquad +\int_{\Omega^{h-l}}\big(\mathbf{u}(x,t)-\mathbf{u}^l(x,t)\big)\otimes\big(\mathbf{b}(x,t)-\mathbf{b}^l(x,t)\big):\big(\theta_{h,l}\nabla\mathbf{b}^l\big) dx.\\[4mm]
\end{array}
\end{equation}

As mentioned before, due to the terms $\int_{\Omega^{h-l}}\left(\mathbf{b}^l\otimes\mathbf{b}^l\right):\left(\theta_{h,l}\nabla\mathbf{u}^l\right)dx$ in \eqref{3-1-23} and $\int_{\Omega^{h-l}}\left(\mathbf{u}^l\otimes\mathbf{b}^l\right):\left(\theta_{h,l}\nabla\mathbf{b}^l\right)dx$ in \eqref{3-1-26}, we can not prove that  $I_{21}(t)$ or $I_{41}(t)$ vanishes as $l\rightarrow 0$ seperately.
However,
adding them together, we obtain
\begin{equation}\label{3-1-28}
\begin{split}
&I_{21}(t)+I_{41}(t)\\
=&-\int_{\Omega^l}\left(\mathbf{b}^l\otimes\mathbf{b}^l\right):\nabla\left(\theta_{h,l}\mathbf{u}^l\right)dx
-\int_{\Omega^l}\left(\mathbf{u}^l\otimes\mathbf{b}^l\right):\nabla\left(\theta_{h,l}\mathbf{b}^l\right)dx\\
=&-\int_{\Omega^l}b_i^lb_j^l(\partial_{x_j}\theta_{h,l})u_i^ldx
-\int_{\Omega^l}\theta_{h,l}b_i^lb_j^l\left(\partial_{x_j}u_i^l\right)dx\\
&-\int_{\Omega^l}u_i^lb_j^l(\partial_{x_j}\theta_{h,l})b_i^ldx
-\int_{\Omega^l}\theta_{h,l}u_i^lb_j^l\left(\partial_{x_j}b_i^l\right)dx\\
=& -2\int_{\Omega^l}b_i^lb_j^l(\partial_{x_j}\theta_{h,l})u_i^ldx
-\int_{\Omega^l}\theta_{h,l}b_j^l\partial_{x_j}\left(u_i^lb_i^l\right)dx\\
=& -2\int_{\Omega^l}b_i^lb_j^l(\partial_{x_j}\theta_{h,l})u_i^ldx
+\int_{\Omega^l}(\partial_{x_j}\theta_{h,l})b_j^lu_i^lb_i^ldx\\
=& -\int_{\Omega_h\setminus\Omega_{h-l}}\left(\mathbf{b}^l\cdot\nabla\theta_{h,l}\right)\left(\mathbf{u}^l\cdot\mathbf{b}^l\right)dx.
\end{split}
\end{equation}
By using Lemma \ref{2.2.0} and H\"{o}lder's inequality, we obtain from \eqref{3-1-28}
\begin{equation}\label{3-1-29}
\begin{split}
&\int_0^T |I_{21}(t)+I_{41}(t)| dt\\
&=\int_0^T\left|\int_{\Omega_h\setminus\Omega_{h-l}}\left(\mathbf{b}^l\cdot\nabla\theta_{h,l}\right)\left(\mathbf{u}^l\cdot\mathbf{b}^l\right)dx\right|dt\\
                &\leq C\|\mathbf{b}^l\cdot\nabla\theta_{h,l}\|_{L^3(0,T;L^3(\Omega_h\setminus\Omega_{h-l}))}
                \|\mathbf{u}^l\|_{L^3(0,T;L^\infty(\Omega_h\setminus\Omega_{h-l}))}\\
               &\relphantom{=}\|\mathbf{b}^l\|_{L^3(0,T;L^\infty(\Omega_h\setminus\Omega_{h-l}))}|\Omega_h\setminus\Omega_{h-l}|^{1-\frac{1}{3}}\\
&\leq Cl^{-\frac{1}{3}}\|\mathbf{u}\|_{L^3(0,T;L^\infty(\Omega_{2h}))}\|\mathbf{b}\|_{L^3(0,T;L^\infty(\Omega_{2h}))}
\left(l^{\alpha_2}\|\mathbf{b}\|_{L^3(0,T;B_{3,\infty}^{\alpha_2}(\Omega^{\frac{h}{2}}))}+h^{\alpha_2}\right).
\end{split}
\end{equation}
For the term $I_{22}(t)$ and $I_{42}(t)$, we have
\begin{equation}\label{3-1-30}
\begin{split}
&\int_0^T|I_{22}(t)|dt,\int_0^T|I_{42}(t)|dt\\
&\leq C \left[l^{\alpha_1+2\alpha_2-1}\|\mathbf{u}\|_{L^3(0,T;B_{3,\infty}^{\alpha_2}(\Omega^{\frac{h}{2}}))}
\|\mathbf{b}\|^2_{L^3(0,T;B_{3,\infty}^{\alpha_2}(\Omega^{\frac{h}{2}}))}\right.\\
&\left.+l^{-\frac{1}{3}}\left(l^{\alpha_2}\|\mathbf{b}\|_{L^3(0,T;B_{3,\infty}^{\alpha_2}(\Omega^{\frac{h}{2}}))}+h^{\alpha_2}\right)
\|\mathbf{u}\|_{L^3(0,T;L^\infty(\Omega_{2h}))}\|\mathbf{b}\|_{L^3(0,T;L^\infty(\Omega_{2h}))}
\right].
\end{split}
\end{equation}
Combining the estimates \eqref{3-1-29} and \eqref{3-1-30} together, if $\alpha_1,\,\alpha_2>\frac{1}{3}$, we have
\begin{equation}\label{3-1-31}
\int_0^T |I_2(t)+I_4(t)| dt\rightarrow 0, \,\, {\rm as} \,l\rightarrow 0.
\end{equation}

\underline{\bf{Step 5}}. We estimate the pressure term $I_5(t)$. Using the divergence-free property of $\mathbf{u}^l$, we can rewrite the pressure term as
\begin{equation}\label{3-1-32}
\begin{split}
I_5(t)&=-\int_{\Omega^l}\nabla\pi^l\cdot\theta_{h,l}\mathbf{u}^l dx
 =-\int_{\Omega^l}\theta_{h,l}\left(\partial_{x_i}\pi^l\right) u_i^ldx
 =\int_{\Omega^l}\left(\partial_{x_i}\theta_{h,l}\right)\pi^lu_i^ldx\\
 &=\int_{\Omega_h\setminus\Omega_{h-l}}\pi^l\left(\mathbf{u}^l\cdot\nabla\theta_{h,l}\right) dx.
\end{split}
\end{equation}
Using the condition $(ii)$ in Theorem \ref{1.1.}, we deduce
\begin{equation}\label{3-1-33}
\begin{split}
\int_0^T|I_5(t)|dt
=&\int_0^T\left|\int_{\Omega_h\setminus\Omega_{h-l}}\pi^l\mathbf{u}^l\cdot\nabla\theta_{h,l}dx\right|dt\\
\leq& C\left\|\pi^l\right\|_{L^{3/2}(0,T;L^\infty(\Omega_h\setminus\Omega_{h-l}))}\left\|\mathbf{u}^l\cdot\nabla\theta_{h,l}\right\|_{L^3(0,T;L^3(\Omega_h\setminus\Omega_{h-l}))}
\left|\Omega_h\setminus\Omega_{h-l}\right|^{1-\frac{1}{3}}\\
\leq& Cl^{-\frac{1}{3}}\left\|\pi\right\|_{L^{3/2}(0,T;L^\infty(\Omega_{2h}))}
\left(l^{\alpha_1}\|\mathbf{u}\|_{L^3(0,T;B_{s,\infty}^{\alpha_1}(\Omega^{\frac{h}{2}}))}+h^{\alpha_1}\right)\rightarrow0, \,\,{\rm as}\,l\rightarrow 0,
\end{split}
\end{equation}
provided that $\alpha_1 >\frac{1}{3}$.

To conclude the proof of energy conservation, it suffices to combine Steps 2-5 together to get
\begin{equation}\label{3-1-34}
\int_0^T|I(t)|dt\rightarrow 0, \,\,{\rm as}\,l\rightarrow 0,
\end{equation}
provided that $\alpha_1,\,\alpha_2 >\frac{1}{3}$.

Now we prove the cross-helicity conservation. Multiplying the equality \eqref{3-10-1} by $\theta_{h,l}\mathbf{b}^l$ and \eqref{3-10-2} by $\theta_{h,l}\mathbf{u}^l$, and then summing the resulted equalities together and integrating over $\Omega^l$ with respect to $x$, we have
\begin{equation}\label{3-201}
\int_0^t\left(\int_{\Omega^l}\theta_{h,l}\mathbf{u}^l\cdot\mathbf{b}^ldx\right)\chi'(t)dt
+\int_0^t K(t)\chi(t)dt=0,
\end{equation}
where
\begin{equation}\label{3-202}
\begin{split}
K(t)=&-\int_{\Omega^l}div(\mathbf{u}\otimes\mathbf{u})^l\cdot\theta_{h,l}\mathbf{b}^ldx
+\int_{\Omega^l}div(\mathbf{b}\otimes\mathbf{b})^l\cdot\theta_{h,l}\mathbf{b}^ldx\\
&-\int_{\Omega^l}div(\mathbf{b}\otimes\mathbf{u})^l\cdot\theta_{h,l}\mathbf{u}^ldx
+\int_{\Omega^l}div(\mathbf{u}\otimes\mathbf{b})^l\cdot\theta_{h,l}\mathbf{u}^ldx
-\int_{\Omega^l}\nabla\pi^l\cdot\left(\theta_{h,l}\mathbf{b}^l\right) dx.
\end{split}
\end{equation}
Similar to the proof of energy conservation, our main task here is to prove that
\begin{equation}\label{3-777}
K(t)\rightarrow 0, \,\,{\rm in} \,L^1(0,T), \,\, {\rm as} \,l\rightarrow 0.
\end{equation}
Since estimates for each term in \eqref{3-202} are analogous to that in $I(t)$, we omit them here.\;\;\;\;\;\;$\Box$

%\section{References}

\end{document}